\documentclass{article}

\usepackage[square,sort&compress,colon,numbers]{natbib}
\usepackage[utf8]{inputenc}
\usepackage{amsthm,amsmath,amsfonts,amssymb,array, graphicx, epsfig, fancyhdr, stmaryrd,algpseudocode}
\usepackage{mathrsfs} 
\usepackage{pgfplots}
\usepackage{tikz}
\usepackage{xcolor}
\usepackage{comment}
\usepackage[caption=false]{subfig}
\definecolor{backcolor}{rgb}{.7,.7,1}
\definecolor{backcolor2}{rgb}{1,.7,0.7}
\usetikzlibrary{positioning}
\usepackage{enumitem}
\usepackage{hyperref}

\usetikzlibrary{shapes,backgrounds,calc}

\usetikzlibrary{decorations.pathreplacing}
\usetikzlibrary{arrows.meta}
\pgfplotsset{compat=1.17}
\makeatletter
\tikzset{circle split part fill/.style  args={#1,#2}{%
 alias=tmp@name, 
  postaction={%
    insert path={
     \pgfextra{%
     \pgfpointdiff{\pgfpointanchor{\pgf@node@name}{center}}%
                  {\pgfpointanchor{\pgf@node@name}{east}}%
     \pgfmathsetmacro\insiderad{\pgf@x}
      \fill[#1] (\pgf@node@name.base) ([xshift=-\pgflinewidth]\pgf@node@name.east) arc
                          (0:180:\insiderad-\pgflinewidth)--cycle;
      \fill[#2] (\pgf@node@name.base) ([xshift=\pgflinewidth]\pgf@node@name.west)  arc
                           (180:360:\insiderad-\pgflinewidth)--cycle;            
         }}}}}  
 \makeatother  

\pgfplotsset{select coords between index/.style 2 args={
    x filter/.code={
        \ifnum\coordindex<#1\fi
        \ifnum\coordindex>#2\fi $$L_f v (x) = w(x)\quad v\in D(L_f).$$

    }
}}
\tikzset{
    position/.style args={#1:#2 from #3}{
        at=(#3), anchor=#1+180, shift=(#1:#2)
    }
}
\newtheoremstyle{plainNoItalics}{}{}{\normalfont}{}{\bfseries}{.}{ }{}

\theoremstyle{plain}
\newtheorem{theorem}{Theorem}[section]

\newtheorem{corollary}[theorem]{Corollary}

\newtheorem*{theorem*}{Theorem}

\newtheorem*{lemma*}{Lemma}
\newtheorem*{corollary*}{Corollary}

\newtheorem*{observation*}{Observation}
\newtheorem*{example*}{Example}
\newtheorem*{assumption*}{Assumption}

\theoremstyle{definition}
\newtheorem{definition}{Definition}
\newtheorem{remark}[definition]{Remark}

\theoremstyle{plain}

\newcommand{\gr}{\mbox{gr}}

\usepackage{amsmath}

\title{Some suggestions concerning the conjecture in: "Tractable semi-algebraic approximation using Christoffel-Darboux kernel"}
\author{Mathias Oster
\thanks{Technische Universit\"at Berlin,
            Strasse des 17. Juni 135,
            10623 Berlin, Germany,
  oster@math.tu-berlin.de }
 \and
Reinhold Schneider
\thanks{Technische Universit\"at Berlin,
            Strasse des 17. Juni 135, 
            10623 Berlin, Germany,
  schneidr@math.tu-berlin.de} }
\date{March 2022}

\begin{document}
\maketitle
\begin{abstract}
    In "Tractable semi-algebraic approximation using Christoffel-Darboux kernel"  Marx, Pauwels, Weisser, Henrion and Lasserre conjectured, that the approximation rate $\mathcal O (\frac 1 {\sqrt d})$ of a Lipschitz functions by a semi-algebraic function induced by a Christoffel-Darboux kernel of degree $d$ in the $L^1$ norm can be improved for more regular functions. Here we will show, that for semi-algebraic and definable functions the results can be strengthened to a rational approximation rate in the $L^\infty$ norm.
\end{abstract}

\section{Introduction}

Approximating functions based on the data of their graphs is a fruitful and hard undertaking. In \cite{PauwelsMarx} was suggested, that one can reconstruct the graph of a function as the set of global minima  of the inverse of the Christoffel function associated to the measure $\mu$ induced by the graph. In particular, let $\Pi_d$ the multi-polynomials of degree at most $d$ in $n+1$ variables and let $\{b_i, i\in\{1,\dots,N\}, b_i\in \Pi_d\}$ be a basis of $\Pi_d$, i.e. $N=\begin{pmatrix}n+d\\n\end{pmatrix}$. Then by denoting $[b_i]$ as a vector with the basis functions as components one gets that 
$$[M_{\mu,d}]_{ij} = \int_\Omega [b_i(x,y)]_i [b_j(x,y)]_jd\mu(x,y) = \int_\Omega [b_i(x,f(x))]_i[b_j(x,f(x))]_jdx$$ 
is the moment matrix associated to the degenerate measure $\mu$ induced by the graph of $f$. As the measure is degenerate, the moment matrix is not invertible. Applying a Tikhonov regularisation $M_{\mu,\beta,d} = M_{\mu,d}+\beta I = M_{\mu+\beta\mu_0}$ one can define
$$q_{\beta,d}(x) = \min_{y\in \mathbb R} [b_i(x,y)]^T M_{\mu,\beta,d}^{-1}[b_i(x,y)].$$
In \cite{PauwelsMarx}, it was shown, that for Lipschitz functions one can approximate $f$ by $q_{\beta,d}$ with rate $\mathcal O(\frac 1 {\sqrt{d}})$ in $L^1(\Omega)$ for some compact $\Omega$.

In \cite{Pauwels2021} the Christoffel function of the degenerate measure was characterized as  $$\Lambda_{d} (x,y) = \inf_{P\in \Pi_d, P(x,y)=1} \int_\Omega P(z)^2 dz =\begin{cases}0 \quad \text{if } [b_i(x,y)]_i \in ker(M_f),\\ ([b_i(x,y)]_iM^\dagger[b_i(x,y)]_i)^{-1}\quad \text{else}.\end{cases} $$
If $(x,y)\notin V $ then $[b_i(x,y)]\in ker(M_f)$. Furthermore, there is a degree $d_*>0$ such that for all $d\geq d_*$ it holds, that if $\Lambda_d(x,y)=0$ then $(x,y)\neq V$. As $q_d$ can be written in terms of the Christoffel function as
$$q_d(x) = \frac 1{\displaystyle\max_{y\in\mathbb R} \Lambda_d(x,y)}$$
we will focus on the Christoffel function of a degenerate measure.

In this paper are some arguments, that for a semi-algebraic function $f$ one can find a degree $d$ such that the approximation induced by the inverse of the Christoffel function is exact except on a tubular neighbourhood of the critical set of $f$ and its derivatives. Furthermore, as semi-algebraic functions can apoproximate definable functions in way similiar to the approximation of $C^d$ functions by polynomials, we get a rational  approximation rate in the $L^\infty$ norm for definable functions.

\section{Notation and preliminaries}
Let $\gr f = \{(x,f(x)):x\in X\}$ be the graph of $f$, $\overline{U}$ be closure of $U$, $\partial U$ the boundary  of $U$ and $U^o = \overline{U}\setminus\partial U$ the interior of $U$. 

Let $\mathbb B_\delta(x)$ denote the ball around $x$ with radius $\delta$. Furthermore, let 
\begin{align*}
    reach(X) &= \sup \{r \in \mathbb{R}: 
             \forall  x \in \mathbb{R}^n\setminus X\text{ with }{\rm dist}(x,X) < r \\
              &\text{exists a unique closest point }y \in X\text{ s.t. }{\rm dist}(x,y)= {\rm dist}(x,X)\}.
\end{align*}
Recall, that a semi-algebraic set is finite collection of polynomial equalities and inequalities.

\section{Main results: Semi-algebraic functions}

Let $f:\mathbb R^n\supset X \to\mathbb R$ be a semi-algebraic function with semi-algebraic domain $X$, such that $\gr f \subset \overline{\gr f ^o}$, i.e. we avoid isolated points of the graph. There is a finite decomposition of $X$ in semi-algebraic cells $A_i$ such that $f|_{A_i}$ is analytic, \cite{SemiAlg-Analyt}. Furthermore, there is a collection of finitely many sets $C_i\in\{A_i\}$ of lower dimensional submanifolds of $\mathbb R^n$ that are the critical points of $f$ or any derivative. These are precisely the cells with codimension greater than 1. Let $S_{crit} = \cup_{i}C_i$.  
\\

Let $V$ be the Zariski closure of $\gr f$, i.e. intuitively, extend the defining polynomial equations onto all of $X$. Here, the structure of $f$ is important, as for semi-algebraic functions $V$ is actually a true subset of the ambient space.

 For $\delta_1>0$ consider $S_{crit, \delta_1} = \cup_{x\in S_{crit}} \mathbb B_{\delta_1}(x)$ with $S_{crit} = \cup_i \gr(f|_{C_i})$, i.e. we remove some tubular neighborhood of the critical set of $f$. Note, that we can choose $\delta_1$ such that all $A_i$ that are open in $\gr f$ are not contained entirely in $S_{crit,\delta_1}$. (We will see all important bits of the graph). Intuitively, the critical set of the semi-algebraic functions, is where some irreducible components of the Zariski closure intersect. In the neighborhood of this points it is difficult to find an uniform upper and lower bound of the Christoffel function as we will need it later. \\

Now let $\delta_2 \leq reach \ (V\setminus S_{crit, \delta_1})$. Note that for $\delta_1$ small enough we have $\delta_2\geq\delta_1$. 
Then we have, that for any $x\in V\setminus S_{crit, \delta_1}$ that $B_{\delta_1}(x)\cap S_{crit} = \emptyset$.

\begin{theorem}\label{thm_semi_alg}
Let $\delta_1$ as above. Then there is a degree $d$ such that for all $x\in X\setminus S_{crit,\delta_1}$ we have $$f(x) - q_d(x) =0$$ as well as for all $x\in X$ $$\|f-q_d\|_\infty \leq \delta_1.$$
\end{theorem}
\begin{corollary}
If $f$ is a polynomial, then there is a degree $d$ such that $f=q_d$.
\end{corollary}
\begin{proof}[Proof of 
Corollary]
If $f$ is polynomial, then $S_{crit}$ is empty, i.e. we can choose $\delta_1=0$.
\end{proof}

The proof of the theorem will be the rest of this chapter. We will estimate the Christoffel function from above for $(x,y)\in V$ not in the graph of $f$ and then estimate the Christoffel function from below on the graph. We then aim to find a degree $d$ such that these estimates separate between the support of the graph and the rest via the Christoffel function. To this end let $d\mu(x,y) = d\lambda(x)\delta_{f(x)=y}(y)$ be the degenerate measure induced by the graph of $f$.

\paragraph{Upper bounds of the Christoffel function for $(x,y)$ not in the graph of $f$ via Needle Polynomials:} We now we can use the needle polynomials \cite{Kro2013ChristoffelFA} to obtain an upper bound on the Christoffel functions for points in $V$ outside the support of $\mu$. Let $T_d$ be a Chebyshev polynomial of first kind. Then 
$$p_{z,\delta,d}(\tilde z) = \frac{(T_d(1+\delta^2-\|z-\tilde z-y\|^2)}{T_d(1+\delta^2)} $$ 
fulfills $p_{z,\delta,d}(z) = 1, |p_{z,\delta,d}(\tilde z)|\leq 1$ on $\tilde z\in B_1(z)$ and $|p_{z,\delta,d}(\tilde z)|\leq 2^{1-\delta d}$ for $y\in B_1(z)\setminus B_\delta(z)$. 
\\

Consider $\delta_{Max} = \inf \{ \delta >0: \forall x\in \hat S: V\subset B_\delta (x) \}$ where $\hat S$ is the smallest n-dim rectangular containing $V$ (possible since $V$ compact). Notice, that the infimum is actually achieved (again due to compactness). ($\delta_{Max} = 2\cdot diam\ V$)\\

Set $\delta_3 = \frac{\delta_1}{\delta_1 + \delta_{Max}}$
Hence, for $z\in V\setminus supp(\mu)\setminus S_{crit,\delta_1}$ we have
\begin{align*}
    \Lambda_d (z)  
    &= \min_{P\in \Pi_d, P(z)=1} \int P^2(\tilde z)d\mu(\tilde z) \leq \int p_{z,\delta_3,\lfloor \frac d 2 \rfloor}^2(\tilde z)d\mu(\tilde z)\\
    &\leq 2^{2-2\delta_3 \lfloor \frac d 2 \rfloor}\mu(V) = 2^{2-2\delta_3 \lfloor \frac d 2 \rfloor}.
\end{align*}

\paragraph{Lower bounds of the Christoffel function on the graph of $f$:} After we gave an upper bound of the Christoffel function outside the support, we will now lower bound on the support. To this end, let us consider the familiy of meausures

$$d\mu_\epsilon(z) = d\mu_\epsilon(x,y) = d\lambda(x)\delta_\epsilon(x,y)dy$$

where $\delta_\epsilon(x,y) = \frac{1}{\sqrt{2\pi}\ \epsilon\ erf(1)}e^{-\frac{(f(x)-y)^2}{\epsilon^2}}\chi|_{X\times[f(x)-\epsilon,f(x)+\epsilon]}$.\\

As $\delta_\epsilon$ is a Dirac sequence uniformly in $x$ it holds for every $f$ that

$$\lim_{\epsilon\to 0} \int f d\mu_\epsilon = \int f d\mu$$
and therefore especially

\begin{align*}
    \lim_{\epsilon\to 0} \Lambda_{d,\epsilon}(x,y) &=\lim_{\epsilon\to 0} \min_{P,P(x,y)=1} \int P^2(z)d\mu_\epsilon(z) \\
    &\leq \lim_{\epsilon\to 0} \int Q^2(z) d\mu_\epsilon(z)  = \int Q^2(z)d\mu(z)
\end{align*}

for all $Q$ with $Q(x,y) =1$. Hence
$$\lim _{\epsilon\to 0} \Lambda_{\epsilon,d}(z) \leq \Lambda_d(z).$$

Define $\delta_\epsilon^- = \min_{z\in X\times[-\epsilon,\epsilon]}\delta_\epsilon(z) = \frac{1}{\sqrt{2\pi}erf(1)\ e}$. It holds that for $z = (x,f(x))$ with $d(x,\partial X)>\delta$ we have

\begin{align*}
    \Lambda_{d,\epsilon}(z) &= \min_{P(z)=1 } \int_{X\times [-\epsilon,\epsilon]} P^2(\tilde z) d\mu_\epsilon(\tilde z) \geq \delta_\epsilon^-  \min_{P(z)=1} \int_{X\times [-\epsilon,\epsilon]} P^2(\tilde z)d\tilde z \\
    &= \delta_\epsilon^- vol(X\times [-\epsilon,\epsilon]) \Lambda_{box}(z)\geq \delta_\epsilon^- vol(X\times [-\epsilon,\epsilon]) \Lambda_{unit\ box}(0) .
\end{align*}

It turns out, that $\Lambda_{unit\ box,d}(0) = (\sum_{i=0}^d (P_i^2(0))^n)^{-1}$, where $P_i(x)$ is the $i$-th Legendre polynomial and $$P^2_{2i}(0) = (-1)^i\frac{1\cdot 3\cdots(2i-1)}{2\cdot4\cdots 2i} = \frac{\Gamma(k+\frac 1 2)^2}{\pi \Gamma(k+1)^2}$$ and $P_{2i+1}(0)=0$. 
It holds that 
\begin{align*}
    &\sum_{k=0}^d \left(\frac{\Gamma(k+\frac 1 2)^2}{\pi \Gamma(k+1)^2}\right)^n\leq \left(\sum_{k=0}^d \frac{\Gamma(k+\frac 1 2)}{\pi \Gamma(k+1)}\right)^{2n} \\
    &= \left(\frac{(2 (1 + d) \Gamma(\frac 3 2 + d))}{\pi \Gamma(2 + d)}\right)^{2n}\leq (1+d)^{2n}.
\end{align*}

Altogether, we get

$$ \Lambda_{d,\epsilon}(z) \geq \frac{1}{\sqrt{2\pi}\ \epsilon\ erf(1)e}vol(X)2\epsilon \frac 1{(1+d)^{2n}} = \frac{2vol(X)}{\sqrt{2\pi}\ erf(1)e(1+d)^{2n}}.$$

As this is independent of $\epsilon$ we get $$\Lambda_d(z)\geq \frac{2vol(X)}{\sqrt{2\pi}\ erf(1)e(1+d)^{2n}}. $$

Therefore, let $x\in X\setminus Pr(S_{crit,\delta_1})$ such that $d(x,\partial X)\geq \delta_1$. Then 

$$\Lambda_d(x,f(x) \geq \frac{2vol(X)}{\sqrt{2\pi}\ erf(1)e(1+d)^{2n}}$$
and 
$$\Lambda_d(x,y) \leq 2^{2-2\delta_3 \lfloor \frac d 2 \rfloor}$$ where $(x,y)\in V\setminus supp\ \mu\setminus S_{crit,\delta_1}.$

Now we want to find some $d_0$ such that for all $d\geq d_0$ we have that for all $x\in X$ there holds $\Lambda_d(x,y)< \Lambda(x,f(x))$ for $y$ such that $(x,y)\in V.$

This amounts to find a $d_0$ such that for all $d\geq d_0$ we have

\begin{equation}\label{eq::d_ineq}
    \frac{(1+d)^{2n}}{2^{2\delta_3 \lfloor \frac d 2 \rfloor}} < \frac{vol(X)}{2\sqrt{2\pi}\ erf(1)\ e }.
\end{equation}

One can see that $\displaystyle\lim_{d\to\infty}\frac{(1+d)^{2n}}{2^{2\delta_3 d}}=0$ and therefore, there is $d_0$ that we wanted to find.

As a reminder, $d_*>0$ is the degree such that for all $d\geq d_*$ it holds, that if $\Lambda_d(x,y)=0$ then $(x,y)\neq V$. All together we can approximate $f$ exact except from some $\delta$ neighborhood of the critical points of $f$ or its derivative. Notice, that discontinuities are actually no problem. Therefore, for $d\geq\max\{d_0,d_*\}$ we have that for $x\in X\setminus S_{crit,\delta_1}$ it holds that $\Lambda_d(x,y)<\Lambda_d(x,f(x))$ where $(x,y)\in V\setminus supp(\mu)$. This concludes the proof.

\begin{remark}
Let $C =\frac{vol(X)}{2\sqrt{2\pi}\ erf(1)\ e } $. One can go into more detail and rephrase equation \eqref{eq::d_ineq} as
\begin{align*}
    2p\log_2 (1+d) < \log_2(C) +2\delta_3 \lfloor \frac d 2 \rfloor\\
    \Longleftrightarrow \log_2( \frac{1}{C^{\frac 1{2n}}}(1+d))<\frac{2\delta_3}{2n}\lfloor \frac d 2 \rfloor.
\end{align*}
For all $d_0>0$ one has for all $d\geq d_0$ 

$$ \log_2( \frac{1}{C^{\frac 1{2n}}}(1+d))\leq \frac{C^{\frac 1{2n}}}{(1+d_0)}d- \log_2(\frac{1}{C^{\frac 1{2n}}}).$$

Then one can choose $d_0$ such that $\frac{C^{\frac 1{2p}}}{(1+d_0)}< \frac{\delta_3}{n} $, e.g. $d_0= \lceil \frac{C^{\frac 1 {2n}}n}{\delta_3}\rceil-1$. Then there is $d_1$ such that $$\frac{C^{\frac 1{2n}}}{(1+d_0)}d- \log_2(\frac{1}{C^{\frac 1{2n}}})< \frac{\delta_3}{n}d$$ for all $d\geq d_1$. It holds that $d_1 = \left( \frac{C^{\frac{1}{2n}}}{1+d_0}-\frac{\delta_3}{n} \right)^{-1}\log_2(\frac{1}{C^{2n}})$. Hence a sufficient condition for \eqref{eq::d_ineq} would be 
\begin{align*}
    d&\geq \left( \frac{C^{\frac{1}{2n}}}{1+d_0}-\frac{\delta_3}{n} \right)^{-1}\log_2(\frac{1}{C^{2n}}) = \frac{(1+d_0)n}{C^{\frac{1}{2n}}n-\delta_3(1+d_0)}\log_2(\frac 1 {C^{\frac 1 {2n}}})\notag\\
    &= \frac{n}{\frac{C^{\frac{1}{2n}}n}{\lceil \frac{C^{\frac 1 {2n}}n}{\delta_3}\rceil}-\delta_3}\log_2(\frac 1 {C^{\frac 1 {2n}}}).
\end{align*}
As $C$ is independent of $\delta_3$ and by fixing $n$, one can define some $0<\epsilon\leq 1$ such that 
\begin{equation*}
    \frac{C^{\frac{1}{2n}}n}{\lceil \frac{C^{\frac 1 {2n}}n}{\delta_3}\rceil}-\delta_3 = (\epsilon-1)\delta_3
\end{equation*}
and hence
\begin{equation}\label{eq::deltaD}
    d\geq \frac{n}{(\epsilon-1)\delta_3}\log_2(\frac 1{ C^{\frac{1}{2n}}}).
\end{equation}
\end{remark}
\paragraph{What happens inside the tubular neighborhood of the critical points?}
Note that, as the inequality can be made strict and the Christoffel function is continuous on $V$, we have, that the argmax of the Christoffel function in a neighborhood of points on the boundary of the tube are still on the support of the measure.

\section{Corollary results: Definable Functions}
Definable functions are a generalization of semi-algebraic functions. Especially, there is a similar cell decomposition possible. Let us fix an o-minimal structure, see e.g. \cite{Coste2002ANIT}. If $f:X\to\mathbb R$ is a definable function (also called tame function) with definable domain $X$ then there exists a decomposition of $X$ in $C^t$ - cells $A_i$ such that $f|_{A_i}$ is $C^t$. 

We can then approximate the defining equations of $A_i$ and $f|_{A_i}$ by polynomials $p|_{A_i}$ of degree $s$ such that $\|f|_{A_i}-p|_{A_i}\| \leq C \frac1{s^t}$ \cite{polapprox}. 

However, the collections of $p|_{A_i}$ define then a semi-algebraic function and one can use the above considerations to find a degree $d$ to approximate this function exact except on the critical points. However, the approximation deviates on this critical points only by some $\delta$ that can be chosen smaller than the desired accuracy of approximating the $C^t$ functions by polynomials. It is important, that the number of cells plays a role in the degree necessary.

\begin{corollary}
Let $X$ be a definable set and $f:X\to \mathbb R$ be a definable function. For any $t,s\in\mathbb N$ there is a degree $d$ and a constant $C>0$ such that 
$$\|f-q_d\|_\infty\leq C \frac 1 {s^t} $$
\end{corollary}
\begin{proof}
There is a cell decomposition of $X$ in $C^t$ definable cells $A_i$ such that $f|_{A_i}$ is also $C^t$, \cite{Coste2002ANIT}. For every cell $A_i$ there is a finite list of defining $C^n$ functions $h_{ij}$. Each $C^t$ function $h_{ij}$ can be approximated by a degree $s$  polynomial $p_{ij}$ such that $\|h_{ij}-p_{ij}\|_\infty \leq C_{ij}\frac 1 {s^t}$. The collection $p_{ij} $ and the order relations of $h_{ij}$ define a semi-algebraic function $p$. Let $\delta\leq \max C_{ij}\frac 1 {s^t}$. Then there is  a $d$ such that the function $p$ approximated by $q_d$ as in theorem \ref{thm_semi_alg}
.\end{proof}
\begin{corollary}
Let $f\in C^k(X)$ for some $k>0$. Then there is a degree $d$ and a constant $C>0$ such that 
$$\|f-q_d\|_{\infty}\leq C\frac 1 {s^k}.$$
\end{corollary}
\begin{proof}
As $f\in C^k(X)$ one can approximate it by an polynomial in $\mathcal O(d^{-k})$. Since polynomials can be recovered exactly, the claim follows.
\end{proof}
\begin{remark}
As can be seen from inequality \eqref{eq::deltaD}, one can bound the dependence of $d_0$ on $s^t$ linearly from above. However, $d_*$ will grow with the number of components and degree of the polynomials describing this components. Note that outside the critical set we have $d\in\mathcal O(s)$. Furthermore, this shows, that no oscillation will take place.
\end{remark}

\section{Christoffel-Darboux approximation form a Data-driven perspective}

From a numerical point of view, it is not reasonable to work with the degenerate measure without any regularization. In \cite{PauwelsMarx} a Tikhonov regularization was suggested, which leads to an interpretation in the measure framework for a suitable choice of basis. 

However, in a data-driven framework, the data to recover the moment matrix by Monte Carlo methods are already distributed according to some density. Usually one assumes this density to be some Gaussian. This also has a nice intuition in the polynomial framework. The Gaussian distribution smooths out the 'steepness' of the Christoffel function stemming from the degenerate measure, but the Christoffel-Darboux kernel of the Gaussian perturbed meausure will very likely have the same set of global minimizers. This 'smoothing of steepness' will render the optimization procedure to actually obtain the funciton approximation numerically more stable. 

Also one can observe, that the lower bounds still hold true, as they are actually designed via Gaussian pertubations. Also the upper bound can be adapted, if we use a Gaussian like bumb function with support in the unit ball. Denoting the set $S_\delta = \{(x,y)\in B_1(0): \|(x,y)-(x,f(x))\|\leq\delta \} $, one gets the same lower bound for all $z\in B_1\setminus S_{\delta} $. Hence, we get the following corollary for the Christoffel-Darboux approximate $q_{d,\epsilon}$ induced by the perturbed measure $\mu_{\epsilon}$.
\begin{corollary}
Let $\delta>0$. Then there is a degree $d$ such that for all $x\in X$ we get $$\|f-q_{d,\epsilon}\|_\infty \leq \delta.$$
\end{corollary}
Note, that if the data is exact, i.e. without noise, one can introduce artificial noise by putting one dimensional white noise to the function value, i.e. one creates new samples $(x_i,y_{ij})$ by setting $y_{ij}=f(x_i)+\xi$ for some normally distributed $\xi$. 
\section{Discussion}
 It turns out, that semi-algebraic functions can be approximated exactly by Christoffel-Darboux kernels on their domain of analyticity. Furthermore, the $L^\infty$ error can be bounded on the domain of irregularity. Lastly, using Gaussian pertubations in an Data-driven setting might be a promising alternative to Tikhonov regularization.
\bibliography{main} 
\end{document}